\documentclass[11pt]{amsart}
\usepackage[T1]{fontenc}
\usepackage{lmodern}
\usepackage{amsmath,amssymb,amsthm,mathtools}
\usepackage[a4paper,margin=30mm]{geometry}
\usepackage[colorlinks=true,linkcolor=blue,citecolor=blue,urlcolor=blue]{hyperref}
\newtheorem{theorem}{Theorem}[section]
\newtheorem{lemma}[theorem]{Lemma}
\newtheorem{corollary}[theorem]{Corollary}
\newtheorem{proposition}[theorem]{Proposition}
\title{Commutative coherent rings with $\mathcal{GP}\not\subseteq\mathcal{GF}$}
\author{Chencheng Zhang}
\thanks{This work was supported by the National Natural Science Foundation of China (No.~12131015).}
\date{\today}
\begin{document}
\begin{abstract}
For every field $k$, we construct a commutative coherent $k$-algebra admitting a Gorenstein projective module that is not Gorenstein flat.
\end{abstract}
\maketitle

\section{Introduction}

Throughout, rings are commutative with identity, ring homomorphisms preserve identities, and modules are unital.
Tensor products (respectively, internal Hom) of complexes use coproduct (respectively, product) totalisation, with the usual Koszul signs.

Gorenstein projective modules extend Auslander--Bridger's theory of modules of $G$-dimension zero over Noetherian rings \cite{AuslanderBridger1969,EnochsJenda1995}.
An $R$-module is \emph{Gorenstein projective} if it is a cycle of an acyclic complex of projective modules which remains acyclic under $\operatorname{Hom}_R(-,W)$ for every projective module $W$.
Enochs, Jenda and Torrecillas introduced the corresponding flat notion \cite{EJT1993}: an $R$-module is \emph{Gorenstein flat} if it is a cycle of an acyclic complex of flat modules which remains acyclic under $E\otimes_R-$ for every injective module $E$.
We denote these classes by $\mathcal{GP}(R)$ and $\mathcal{GF}(R)$, respectively.

Although projective modules are flat, the analogous inclusion $\mathcal{GP}(R)\subseteq\mathcal{GF}(R)$ does not follow from the definitions.
The question was already raised in Holm's 2000 master's thesis \cite[Remark~4.5(3)--(5)]{DingLiMao2009}, where coherent rings were singled out for further investigation.
A positive result of Holm states that $\mathcal{GP}(R)\subseteq\mathcal{GF}(R)$ when $R$ is coherent with finite big finitistic projective dimension \cite[Proposition~3.4]{Holm2004}.
In particular, the inclusion holds for Noetherian rings of finite Krull dimension \cite[Remark~3.3 and Proposition~3.4]{Holm2004}.

Later, \v{S}aroch and \v{S}\v{t}ov\'{i}\v{c}ek introduced \emph{projectively coresolved Gorenstein flat} modules in their study of Gorenstein flat approximation theory \cite{SarochStovicek2020}.
These are the cycles of acyclic complexes of projective modules which remain acyclic after tensoring with any injective module.
They proved that the resulting class $\mathcal{PGF}(R)$ satisfies $\mathcal{PGF}(R)\subseteq\mathcal{GP}(R)\cap\mathcal{GF}(R)$,
and moreover, $\mathcal{PGF}(R)$ is the left-hand class of a complete hereditary cotorsion pair \cite[Theorems~4.4 and~4.9]{SarochStovicek2020}.
Iacob established equivalent criteria \cite[Theorem~3]{Iacob2020} that 
\[
\mathcal{GP}(R)=\mathcal{PGF}(R)
\quad\Longleftrightarrow\quad
\mathcal{GP}(R)\subseteq\mathcal{GF}(R).
\]

We show for commutative coherent algebras over an arbitrary field that, Gorenstein projective modules need not be Gorenstein flat.

\begin{theorem}[Main theorem]\label{thm:introduction-main}
For every field $k$, there exists a commutative coherent $k$-algebra $A$ admitting a Gorenstein projective module which is not Gorenstein flat.
\end{theorem}

The construction and proof are elementary; the only external input is a cohomological vanishing theorem of Dicks, Kropholler, Leary and Thomas \cite[Corollary~6.9]{DKLT}.

\section{Proof of the main theorem}

Let $k$ be an arbitrary field and $p$ a prime number distinct from $\operatorname{char}(k)$.
Put $\aleph_\omega=\sup_{n<\omega}\aleph_n$, where $\omega$ is the first infinite ordinal.
Let $G = (\mathbb F_p)^{(\aleph_{\omega})}$ be an Abelian group of $\aleph_{\omega}$ copies of coproduct of $\mathbb F_p$.
Let $R = kG$ be the group algebra with augmentation map
\[
\varepsilon : R \longrightarrow k, \qquad \sum_{g \in G} c_gg \longmapsto \sum_{g \in G} c_g, \qquad (\text{finite sum}).
\]
Put $\widetilde A = k[x,y]/(x,y)^2$, and let $A = \widetilde A \otimes _k R = R[x,y]/(x,y)^2$ be the final ring.
Define $P = \widetilde{A}^{(\aleph_0)}$, and $d : P \to P$ via $d(e_{2r}) = x e_r$ and $d(e_{2r+1}) = y e_r$.
Then $P_\bullet = [ \cdots \xrightarrow d  P \xrightarrow d  P \xrightarrow d  \cdots]$ is a one periodic complex.
Finally fix $Q_\bullet \to k$ to be a projective resolution of $R$-modules.
Set $K_\bullet = P_\bullet \otimes _k Q_\bullet$.
We will show that $Z_0(K_\bullet) \in \mathcal{GP}(A) \setminus \mathcal{GF}(A)$.

\begin{lemma}\label{lem_RvnR}
    $R$ is von Neumann regular, thus is coherent.
\end{lemma}

\begin{proof}
Fix arbitrary $a\in R$.
The support of $a$ is contained in a finite subgroup $H\subseteq G$.
Since $|H|$ is invertible in $k$, Maschke's theorem implies that $kH$ is semisimple.
Since $H$ is abelian, $kH$ is a finite product of fields.
Hence there is $b\in kH$ with $a=aba$.

Hence $R$ is von Neumann regular.
Consequently $R$ is coherent.
\end{proof}

\begin{corollary}\label{cor_Rkk}
    $k$ is a flat $R$-module, and $\mathrm{Hom}_k(R,k) \otimes _R k \neq 0$.
\end{corollary}
\begin{proof}
Any module over a von Neumann regular ring is flat.
    Since $\varepsilon$ is surjective, $k \cong \mathrm{Hom}_k(k,k) \xrightarrow{\mathrm{Hom}_k(\varepsilon,k)}\mathrm{Hom}_k(R,k)$ is injective.
    Note that $k$ is flat.
    Hence $k \otimes _R k$ identifies a subset of $\mathrm{Hom}_k(R,k) \otimes _Rk$.
    Consequently $\mathrm{Hom}_k(R,k) \otimes _R k \neq 0$.
\end{proof}

\begin{proposition}[{\cite[Corollary~6.9]{DKLT}}]
For every abelian group $U$ and every $n\geq0$, one has $H^n(G,\mathbb ZG\otimes_{\mathbb Z}U)=0$, where $G$ acts as $g \otimes 1$ on $\mathbb ZG \otimes _{\mathbb Z} U$.
\end{proposition}

\begin{proof}
Since $|G|=\aleph_\omega$, the Abelian group $G$ has $\aleph$-rank \cite[Notation 1.1]{DKLT} $\omega$.
By \cite[Corollary~6.9]{DKLT}, one has $H^n(G,\mathbb ZG\otimes_{\mathbb Z}U)=0$ whenever $n<\omega+1$.
\end{proof}

\begin{corollary}\label{cor_ext_vanishing}
For every set $J$ and every $n\geq0$, one has $\operatorname{Ext}_R^n(k,R^{(J)})=0$.
\end{corollary}

\begin{proof}
Clearly $\operatorname{Ext}_R^n(k,R^{(J)})\cong H^n(G,R^{(J)}) \cong H^n(G,\mathbb ZG\otimes_{\mathbb Z}k^{(J)})=0$.
\end{proof}

\begin{lemma}
The one periodic complex
    $P_\bullet = [ \cdots \xrightarrow d  P \xrightarrow d  P \xrightarrow d  \cdots]$ is acyclic.
\end{lemma}

\begin{proof}
Every element of $P$ is a unique finite sum $v=\sum_{r\geq0}(a_r+b_rx+c_ry)e_r$.
Since $(x,y)^2=0$, one has $d(v)=\sum_{r\geq0}(a_{2r}x+a_{2r+1}y)e_r$.
Hence $d(v)=0$ if and only if $a_r=0$ for every $r$.
Hence $\ker d=(x,y) \cdot P =\operatorname{im}d$.
Consequently $P_\bullet$ is acyclic.
\end{proof}

\begin{lemma}
Each term of 
    $K_\bullet = P_\bullet \otimes _k Q_\bullet$ is a projective $A$ module.
\end{lemma}

\begin{proof}
    Since $\widetilde A \otimes _k R \cong A$, taking direct summands and coproducts complete the proof.
\end{proof}

\begin{lemma}\label{lem_isGP}
$K_\bullet$ and $\operatorname{Hom}_A(K_\bullet,W)$ are acyclic for any projective $A$-module $W$.
\end{lemma}

\begin{proof}
Since $K_\bullet$ is contractible in $\mathrm{Mod}_k$, it is acyclic in $\mathrm{Mod}_A$.

It suffices to show the second assertion for an arbitrary free module $W=A^{(J)}$.
By Corollary~\ref{cor_ext_vanishing}, $\operatorname{Hom}_R(Q_\bullet,R^{(J)})$ is acyclic.
Since $\widetilde A$ is finite-dimensional over $k$, there is a natural isomorphism of complexes $\operatorname{Hom}_R(Q_\bullet,A^{(J)})\cong \operatorname{Hom}_R(Q_\bullet,R^{(J)}) \otimes_k\widetilde A$.
Since $\operatorname{Hom}_R(Q_\bullet,R^{(J)}) $ is contractible over $k$, $\operatorname{Hom}_R(Q_\bullet,A^{(J)})$ is contractible over $\widetilde A$.
Thus
\[
\operatorname{Hom}_A(K_\bullet,A^{(J)})\cong
\operatorname{Hom}_{\widetilde A}^{\bullet} \left(P_\bullet,\operatorname{Hom}_R(Q_\bullet,A^{(J)})\right) \quad \text{is acyclic}. \qedhere.
\]  
\end{proof}

\begin{lemma}\label{lem_notGF}
$\operatorname{Hom}_k(A,k)$ is an injective $A$ module, while $\operatorname{Hom}_k(A,k)\otimes_AK_\bullet$ is not acyclic.
\end{lemma}

\begin{proof}
Since $\operatorname{Hom}_k(A,-)$ has an exact left adjoint, it preserves injective modules.
Hence $\operatorname{Hom}_k(A,k)$ is injective.
Let $\lambda_1,\lambda_x,\lambda_y$ be the basis of $\operatorname{Hom}_k(\widetilde A,k)$ dual to $1,x,y$.
We claim that $\lambda_y\otimes e_0$ is a cycle but not a boundary in $\operatorname{Hom}_k(\widetilde A,k)\otimes_{\widetilde A}P_\bullet$.
\begin{enumerate}
    \item Since $(1\otimes d)(\lambda_y\otimes e_0)=\lambda_y\otimes xe_0=(\lambda_yx)\otimes e_0=0$, $\lambda_y\otimes e_0$ is a cycle.
    \item Every element of $\operatorname{Hom}_k(\widetilde A,k)\otimes_{\widetilde A}P$ has a unique expression $\sum_{r\geq0}\mu_r\otimes e_r$.
    Then 
        \[
        (1\otimes d)\bigg(\sum_{r\geq0}\mu_r\otimes e_r\bigg)=\sum_{r\geq0}(\mu_{2r}x+\mu_{2r+1}y)\otimes e_r =\sum_{r\geq0}(\mu_{2r}(x)\lambda_1+\mu_{2r+1}(y)\lambda_1)\otimes e_r.
        \]
    Since every boundary has all its coefficients in $k\lambda_1$, $\lambda_y\otimes e_0$ is not a boundary.
\end{enumerate}
Hence $H_0(\operatorname{Hom}_k(\widetilde A,k)\otimes_{\widetilde A}P_\bullet)\ne0$.
Since $\widetilde A$ is finite-dimensional over $k$, one has
\[\operatorname{Hom}_k(A,k)\cong\operatorname{Hom}_k(\widetilde A,k)\otimes_k\operatorname{Hom}_k(R,k).
\]
Consequently, $\operatorname{Hom}_k(A,k)\otimes_AK_\bullet\cong\left((\operatorname{Hom}_k(\widetilde A,k)\otimes_{\widetilde A}P_\bullet)\otimes_k(\operatorname{Hom}_k(R,k)\otimes_RQ_\bullet)\right)$.
Since $R$ is von Neumann regular, the $R$-module $\operatorname{Hom}_k(R,k)$ is flat, and hence
\[
H_j(\operatorname{Hom}_k(R,k)\otimes_R Q_\bullet)
\cong
\begin{cases}
\operatorname{Hom}_k(R,k)\otimes_R k,&j=0,\\
0,&j\ne0.
\end{cases}
\]
Hence the K\"unneth formula (over $k$)  gives
\begin{align*}
H_0(\operatorname{Hom}_k(A,k)\otimes_A K_\bullet)
&\cong H_0\left(\big((\operatorname{Hom}_k(\widetilde A,k)\otimes_{\widetilde A}P_\bullet)\otimes_k(\operatorname{Hom}_k(R,k)\otimes_RQ_\bullet)\big)\right)\\
&\cong H_0(\operatorname{Hom}_k(\widetilde A,k)\otimes_{\widetilde A}P_\bullet)\otimes_k\bigl(\operatorname{Hom}_k(R,k)\otimes_R k\bigr)\quad \ne0.
\end{align*}
Here $H_0(\operatorname{Hom}_k(\widetilde A,k)\otimes_{\widetilde A}P_\bullet) \neq 0$ is above. Corollary~\ref{cor_Rkk} gives $\operatorname{Hom}_k(R,k)\otimes_R k \neq 0$.
\end{proof}

\begin{proof}[Proof of the main theorem]
Since $A$ is finite free over the coherent ring $R$ (Lemma~\ref{lem_RvnR}), $A$ is coherent.
By Lemma~\ref{lem_isGP}, $Z_{-2}(K_\bullet) \in \mathcal{GP}(A)$.
Meanwhile, the projective resolution
\[
\cdots\longrightarrow K_1\longrightarrow K_0\longrightarrow K_{-1}\longrightarrow Z_{-2}(K_\bullet)\longrightarrow0
\]
gives $\operatorname{Tor}_1^A(\operatorname{Hom}_k(A,k),Z_{-2}(K_\bullet))
\cong H_0(\operatorname{Hom}_k(A,k)\otimes_A K_\bullet)
\ne0$ by Lemma~\ref{lem_notGF}.
Since $\operatorname{Hom}_k(A,k)$ is injective, $Z_{-2}(K_\bullet)$ is not Gorenstein flat.
\end{proof}

\section*{Acknowledgements and reflections on AI-assisted research}

The author is grateful to his advisor, Pu Zhang, for urging him to think further about this problem; this encouragement was instrumental in finding the elementary approach presented here.
The author also thanks Junhong Chen, Jian Liu, Xue-Song Lu Tianyang Sun (and JIUCHONG team at the University of Science and Technology of China) for their help with his earlier attempts to solve this problem.

The author thanks ChatGPT for finding complicated counterexamples in earlier arXiv versions.
Had the author not returned to think carefully about the problem, the early counterexample might have remained one that few were willing to read: the problem might have been regarded as ``solved'', yet remained shrouded in fog.


\begin{thebibliography}{99}

\bibitem{DKLT}
W.~Dicks, P.~H. Kropholler, I.~J. Leary and S.~Thomas, \emph{Classifying spaces for proper actions of locally-finite groups}, J. Group Theory \textbf{5} (2002), no.~4, 453--480.
\bibitem{AuslanderBridger1969} M.~Auslander and M.~Bridger, \emph{Stable Module Theory}, Memoirs of the American Mathematical Society, no.~94, American Mathematical Society, Providence, RI, 1969.

\bibitem{EnochsJenda1995} E.~E. Enochs and O.~M.~G. Jenda, \emph{Gorenstein injective and projective modules}, Math. Z. \textbf{220} (1995), 611--633.

\bibitem{EJT1993} E.~E. Enochs, O.~M.~G. Jenda, and B.~Torrecillas, \emph{Gorenstein flat modules}, Nanjing Daxue Xuebao Shuxue Bannian Kan \textbf{10} (1993), 1--9.

\bibitem{SarochStovicek2020} J.~Šaroch and J.~Šťovíček, \emph{Singular compactness and definability for $\Sigma$-cotorsion and Gorenstein modules}, Selecta Math. (N.S.) \textbf{26} (2020), no.~2, Paper No.~23.

\bibitem{Iacob2020} A.~Iacob, \emph{Projectively coresolved Gorenstein flat and Ding projective modules}, Comm. Algebra \textbf{48} (2020), no.~7, 2883--2893.

\bibitem{DingLiMao2009} N.~Ding, Y.~Li, and L.~Mao, \emph{Strongly Gorenstein flat modules}, J. Aust. Math. Soc. \textbf{86} (2009), no.~3, 323--338.

\bibitem{Holm2004} H.~Holm, \emph{Gorenstein homological dimensions}, J. Pure Appl. Algebra \textbf{189} (2004), 167--193.

\end{thebibliography}
\end{document}